\UseRawInputEncoding
\documentclass[12pt]{article}
\usepackage[centertags]{amsmath}
\usepackage{amsfonts}
\usepackage{amssymb}
\usepackage{latexsym}
\usepackage{amsthm}
\usepackage{newlfont}
\usepackage{graphicx}
\usepackage{listings}
\usepackage{booktabs}
\usepackage{abstract}
\usepackage{enumerate}
\usepackage{xcolor}
\RequirePackage{srcltx}
\date{}

\newlength{\defbaselineskip}
\newcommand{\setlinespacing}[1]%
           {\setlength{\baselineskip}{#1 \defbaselineskip}}

\newcommand{\actaqed}{\hfill $\actabox$}
{\medskip\noindent \textit{Proof of #1. }}%
{\actaqed \medskip}

\def\cD{{\mathcal D}}

\def\cL{{\mathcal L}}

\def\cN{{\mathcal N}}

\def\bbN{{\mathbb N}}

\def\bc{\mathbf c}

\def\bx{\mathbf x}

 \def \<{\langle}
\def\>{\rangle}

\def \e{\varepsilon}

\def \ff{\varphi}
\def\al{\alpha}

\def\ga{\gamma}
\def\la{\lambda}

\def \conv{\operatorname{conv}}

\def \sp{\operatorname{span}}

\def \dist{\operatorname{dist}}
\def \Pr{\operatorname{Pr}}
\def \csp{\overline{\operatorname{span}}}

\def \fimt {\varphi^\tau_m}
\def \fmt {f^\tau_m}

\def \ft {f^\tau_{m-1}}
\def \fj {f^\tau_{j-1}}
\def \fijt {\varphi^\tau_j}

\newtheorem{Theorem}{Theorem}[section]
\newtheorem{Lemma}{Lemma}[section]

\newtheorem{Remark}{Remark}[section]

\newtheorem{Corollary}{Corollary}[section]
\numberwithin{equation}{section}

\newcommand{\be}{\begin{equation}}
\newcommand{\ee}{\end{equation}}

\begin{document}
 
 \title{On weak greedy algorithms}

\author{ A. Spivak and  V. Temlyakov}

\newcommand{\Addresses}{{
  \bigskip
  \footnotesize

\medskip
A.S. Spivak, \\ \textsc{Voronezh State University, Russia;\\ Saint Petersburg State University,  Russia.\\
\\ 
E-mail:} \texttt{alexs1nger@yandex.ru}

 \medskip
  V.N. Temlyakov, \textsc{Steklov Mathematical Institute of Russian Academy of Sciences, Russia;\\ Lomonosov Moscow State University, Russia; \\ Moscow Center of Fundamental and Applied Mathematics, Russia.\\  
  \\
E-mail:} \texttt{temlyakovv@gmail.com}

}}

\maketitle
\begin{abstract}
The main goal of this paper is twofold. First, we extend some results 
known in the case of weak greedy algorithms with a scalar parameter  to the case of weak greedy algorithms with a weakness sequence. Second, we formulate a new setting of the problem of convergence of greedy algorithms. 
Usually, we are interested in convergence of an algorithm for all elements of the space. We suggest to 
study convergence of an algorithm for a subset, which   is a generalized octahedron associated with a given dictionary.
\end{abstract}
\section{Introduction}
\label{In}

We begin with some general remarks explaining the goal of this paper and then proceed to the detailed presentation.
We study weak greedy algorithms and focus on the feature {\it weak} of these algorithms. The term weak in the name 
{\it weak greedy algorithm} means that at the most important step of the algorithm -- the greedy step -- we do not look for an optimal element from the given dictionary, which satisfies the optimal greedy inequality, but we are satisfied with an element, which satisfies a weaker than optimal inequality determined but the corresponding weakness parameter. In the general setting the weakness parameter may depend on the iteration of the algorithm. The standard notation for the weakness parameter at the $k$th iteration of the algorithm is $t_k$. For details see the definition of the Weak Greedy Algorithm below.
The weak greedy algorithms with a scalar parameter $t=t_k$, $k=1,2,\dots$, appeared in the very first papers on greedy 
approximation (see \cite{J1}). The weak greedy algorithms with a weakness sequence $\tau := \{t_k\}_{k=1}^\infty$ were introduced and studied in \cite{VT75}. The algorithms with a weakness sequence allow us to consider the following two important special cases: (I) $\tau$ is monotone, $t_1\ge t_2 \ge\dots$, with $t_k \to 0$ as $k\to \infty$, (II) for a given 
subsequence $\cN:=\{n_k\}_{k=1}^\infty$ and a scalar $t\in (0,1]$ the weakness sequence $\tau(\cN,t)$ satisfies the condition $t_n=t$ for $n\in \cN$ and $t_n=0$ otherwise. The main goal of this paper is twofold. First, we extend some results 
known in the case of weak greedy algorithms with a scalar parameter $t$ to the case of weak greedy algorithms with a weakness sequence $\tau$. Second, we formulate a new setting of the problem of convergence of greedy algorithms. 
Usually, we are interested in convergence of an algorithm for all elements, say, of a Hilbert space $H$. We suggest to 
study convergence of an algorithm for a subset $A_1(\cD)$ of $H$, where $A_1(\cD)$ is a generalized octahedron associated with a dictionary $\cD$ (see the definition below). We now proceed to the detailed presentation.

We recall some definitions from the theory of  greedy approximation. 
 Let $H$ be a real Hilbert space with an inner product $\<\cdot,\cdot\>$ and the norm $\|x\|:=\<x,x\>^{1/2}$. We say a set $\cD$ of functions (elements) from $H$ is a dictionary if each $g\in \cD$ has norm one $(\|g\|=1)$
and the closure of $\sp \cD =H.$
We give now the definition of the WGA($\cD,\tau$) (see \cite{VT75} and \cite{VTbook}, pp.81-82). Sometimes for brevity we 
write WGA($\tau$) instead of WGA($\cD,\tau$).

{\bf Weakness sequence.}   Let $\tau := \{t_k\}_{k=1}^\infty$ be a given sequence of numbers $t_k\in [0,1]$, $k=1,2,\dots$. We call such a sequence the {\it weakness sequence}. In the case $\tau = \{t\}$, i.e. $t_k=t$, $k=1,2\dots$, we write $t$ instead of $\tau$ in the notation.

   {\bf Weak Greedy Algorithm  (WGA($\cD,\tau$)).} Let a weakness sequence $\tau = \{t_k\}_{k=1}^\infty$, $0\le t_k \le 1$, be given. We define $f_0^{\tau}:=f$. Then for each $m\ge 1$, we inductively define:
$\fimt \in \cD$ is any satisfying (in the case $t_m=1$ we assume that  such a maximizer exists)
\be\label{P4}
|\<\ft,\fimt\>| \ge t_m \sup_{g\in \cD} |\<\ft,g\>|.
\ee
Then
\be\label{P5}
\fmt := \ft -\<\ft,\fimt\>\fimt
\ee
and
\be\label{P6}
G^\tau_m(f,\cD) := \sum_{j=1}^m \<\fj,\fijt\>\fijt.
\ee

For a dictionary $\cD$ consider the symmetrized dictionary $\cD^\pm := \{\pm g \,:\, g\in \cD\}$ and define $A_1(\cD)$ (generalized octahedron associated with a dictionary $\cD$) to be the closure (in $H$) of the convex hull of $\cD^\pm$. Define
$$
\|f\|_{A_1(\cD)} := \inf \{M>0:\, f/M\in A_1(\cD)\}.
$$
Clearly, $\|f\|\le \|f\|_{A_1(\cD)}$.

The main results of the paper are in Sections \ref{tb} and \ref{BS}. We only formulate one result from Section \ref{tb} (see Corollary \ref{PC1} there), which is devoted to a Hilbert space setting.

\begin{Theorem}\label{InT1} Assume $\tau :=\{t_k\}_{k=1}^\infty$ is a nonincreasing weakness sequence. For any Hilbert space $H$  we have
\be\label{In1}
\|f -G^{\tau}_m(f,\cD)\| \le \left(1+\sum^m_{k=1}t_k^2\right)^{-\al/2} \|f\|^{1-\al} \|f\|_{A_1(\cD)}^\al
\ee
provided $\al \le \frac{t_m}{t_m+2}$.
\end{Theorem}

In the case of a scalar weakness parameter ($\tau =\{t\}$) Theorem \ref{InT1} was proved in \cite{VT196}.  We point out that Theorem \ref{InT1} has an important  new ingredient, which was introduced in the paper \cite{VT196}. Namely, the inequality (\ref{In1})   bounds the error of approximation by the product of both norms -- 
the norm of $f$ and the $A_1(\cD)$-norm of $f$. Typically, only the $A_1(\cD)$-norm of $f$ is used. 

In Section \ref{T} we illustrate the new setting  of  
studying convergence of an algorithm for a subset $A_1(\cD)$ of $H$ instead of the whole space $H$. Theorem \ref{TT1} 
is a simple observation, which gives a criterion for such convergence in a very special case, when $\cD$ is an orthonormal basis. Further discussion of this setting is given in Section \ref{D}. 

It is well known that greedy algorithms in Hilbert spaces are closely related with remote consecutive projections. The reader can find many results and a detailed discussion of the corresponding connection in a very recent nice survey \cite{BK}. 
In Section \ref{P} we present some comments in that direction. In a certain sense Section \ref{P} complements the survey \cite{BK}. For the reader's convenience we give a detailed presentation of the corresponding results. 

{\bf Novelty.} As we already pointed out above one of the main goals of the paper is to extend some results 
known in the case of weak greedy algorithms with a scalar parameter $t$ to the case of weak greedy algorithms with a weakness sequence $\tau$. This is done in Theorem \ref{RCT} (Hilbert spaces), Theorem \ref{BT1} (Banach spaces), and Theorem \ref{PT5a}. The methods of proofs of Theorems \ref{RCT} and \ref{BT1} are similar to the case of scalar weakness parameter and go back to the paper \cite{VT196}. However, the step from scalar weakness parameter to the weakness sequence requires some new technical tricks (see, for instance, inequalities (\ref{4.11}) -- (\ref{4.11b}). For this reason we present detailed proofs of these theorems. We point out that our new results Theorems \ref{RCT} and \ref{BT1} are proved under an extra assumption that the weakness sequence is monotone. We don't know how to get rid of this extra assumption. 
In Section \ref{P} we demonstrate how Theorem \ref{RCT} can be used for 
remote consecutive projections. Theorem \ref{PT7} is a new result, which contains an interesting and important new feature -- the norm of the residual is bounded by the product of two norms, 
the norm of $f$ and the $A_1(\cD)$-norm of $f$. It is the first result of this type in the theory of remote projections. Note that Theorems \ref{D3} and \ref{D4} are direct corollaries of the proofs of known results but they might be of interest and importance. These theorems are in the vein of the second main goal of the paper of studying convergence of an algorithm for a subset $A_1(\cD)$ instead of the whole space $H$. 
 
 \section{Some  results on the rate of convergence of the WGA($\cD,\tau$)} 
\label{tb}

In this section we study a somewhat more general algorithm than the WGA($\cD,\tau$) defined in the Introduction. It is known that the new tuning parameter $b$ helps to obtain better theoretical guarantees for the rate of approximation (see, for example, \cite{VTbook}, Ch.6).
 Let a sequence $\tau = \{t_k\}_{k=1}^\infty$, $0\le t_k \le 1$ and a parameter $b\in (0,1]$ be given.  We define the Weak Greedy Algorithm with parameter $b$.  

{\bf Weak Greedy Algorithm with parameter $b$ (WGA($\cD,\tau,b$))} We define $f_0:=f_0^{\tau,b}:=f$. Then for each $m\ge 1$, we inductively define:

1) $\ff_m:=\varphi^{\tau,b}_m \in \cD$ is any satisfying (in the case $t_m=1$ we assume that  such a maximizer exists)
$$
|\<f_{m-1},\varphi_m\>| \ge t_m \sup_{g\in \cD} |\<f_{m-1},g\>|;
$$

2) 
$$
f_m :=f_m^{\tau,b}:= f_{m-1} -b\<f_{m-1},\varphi_m\>\varphi_m;
$$

3)
$$
G_m(f,\cD):=G^{\tau,b}_m(f,\cD) := b\sum_{j=1}^m \<f_{j-1},\varphi_j\>\varphi_j.
$$
 In the case $t_k = t$, $k=1,2,\dots$, we write $t$ in the notation instead of $\tau$. 
 Note that WGA($\cD,\tau,1$) = WGA($\cD,\tau$).
 
 We proceed to the rate of convergence. The following Theorem \ref{T4.2} was proved in \cite{VT111}.

\begin{Theorem}[{\cite{VT111}}]\label{T4.2} Let $\cD$ be an arbitrary dictionary in $H$. Assume $\tau :=\{t_k\}_{k=1}^\infty$ is a nonincreasing weakness sequence and $b\in(0,1]$. Then for $f \in A_1(\cD)$ we have
\be\label{4.1}
\|f -G^{\tau,b}_m(f,\cD)\| \le e_m(\tau,b)
\ee
where
\be\label{4.1a}
e_m(\tau,b) := \left(1+b(2-b)\sum_{k=1}^m t^2_k\right)^{-\frac{(2-b)t_m}{2(2+(2-b)t_m)}}. 
\ee
\end{Theorem}

Theorem \ref{T4.2} implies the following inequality for any $f$ and any $\cD$
\be\label{4.1b}
\frac{\|f -G^{\tau,b}_m(f,\cD)\|}{\|f\|_{A_1(\cD)}} \le e_m(\tau,b) .
\ee
We now extend Theorem \ref{T4.2} to provide a bound for  
\be\label{4.1c}
 \ga_m^{\tau,b}(\al,H):= \sup_{\cD}\sup_{f\in A_1(\cD),f\neq0}\sup_{G^{\tau,b}_m(f,\cD)}\frac{\|f -G^{\tau,b}_m(f,\cD)\|}{\|f\|^{1-\al} \|f\|_{A_1(\cD)}^\al}  .
\ee
In the case $\tau = \{t_k\}$, $t_k=t$, $k=1,2,\dots$, we write $ \ga_m^{t,b}$ instead of  $\ga_m^{\tau,b}$.

The following Theorem \ref{T4.2ab} was proved in \cite{VT196}. 

\begin{Theorem}[{\cite{VT196}}]\label{T4.2ab} For any Hilbert space $H$   we have
\be\label{4.1eb}
\ga_m^{t,b}(\al,H) \le (1+mb(2-b) t^2)^{-\al/2},   
\ee
provided $\al \le \frac{(2-b)t}{(2-b)t+2}$.
\end{Theorem}

We now prove the following version of Theorem \ref{T4.2ab} with the weakness parameter $t$ replaced by the weakness sequence $\tau = \{t_k\}$. 

\begin{Theorem}\label{RCT} Assume $\tau :=\{t_k\}_{k=1}^\infty$ is a nonincreasing weakness sequence. For any Hilbert space $H$   we have
\be\label{4.1e}
\ga_m^{\tau,b}(\al,H) \le \left(1+b(2-b) \sum^m_{k=1}t_k^2\right)^{-\al/2},   
\ee
provided $\al \le \frac{(2-b)t_m}{(2-b)t_m+2}$.
\end{Theorem}

\begin{proof}  Note that Theorem \ref{T4.2ab} implies Theorem \ref{RCT} with a weaker than (\ref{4.1e}) bound 
\be\label{4.1ec}
\ga_m^{\tau,b}(\al,H) \le (1+b(2-b)m t_m^2)^{-\al/2}.   
\ee
However, bounds (\ref{4.1e}) and (\ref{4.1ec}) are close asymptotically, when $t_m \to 0$.

The proof of this theorem goes along the lines of the proof of Theorem \ref{T4.2} in \cite{VT111}. Let $\cD$ be a dictionary in $H$. We introduce some notations:
$$
f_k:= f^{\tau,b}_k,\quad \varphi_k:=\varphi^{\tau,b}_k, \quad k=0,1,\dots,
$$
$$
a_m := \|f_m\|^2, \quad y_m := 
|\<f_{m-1},\varphi_m\>|, \quad
m=1,2,\dots,  
$$
and consider the sequence $\{B_n\}$ defined as follows
$$
B_0 := \|f\|_{A_1(\cD)},\quad B_m := B_{m-1} +by_m, \quad m=1,2,\dots .
$$
It is clear that $\|f_n\|_{A_1(\cD)}\le B_n$, $n=0,1,\dots$. By Lemma 3.5 from \cite{DT} (see also \cite{VTbook}, p.91, Lemma 2.17) we get
\be\label{4.2}
\sup_{g\in \cD} |\<f_{m-1},g\>| \ge \|f_{m-1}\|^2/B_{m-1}.
\ee
From here and from the equality  
$$
\|f_m\|^2 = \|f_{m-1}\|^2 -b(2-b)\<f_{m-1},\varphi_m\>^2
$$
we obtain the following relations
\be\label{4.3}
a_m = a_{m-1} - b(2-b)y_m^2, 
\ee
\be\label{4.4}
B_m = B_{m-1} +by_m, 
\ee
\be\label{4.5}
y_m \ge t_ma_{m-1}/B_{m-1}. 
\ee
From (\ref{4.3}) and (\ref{4.5}) we obtain
$$
a_m \le a_{m-1} (1-b(2-b)t_m^2a_{m-1}B_{m-1}^{-2}).
$$ 
Using that $B_{m-1}\le B_m$, we derive from here
\be\label{aB}
a_mB_m^{-2} \le a_{m-1}B_{m-1}^{-2}(1-b(2-b)t_m^2a_{m-1}B_{m-1}^{-2}).
\ee
We shall need the
following simple known lemma (see, for instance, \cite{VT211}, Lemma 14.3). 

\begin{Lemma}\label{HL1} Let a number $C_1>0$ and a sequence $\{s_k\}^\infty_{k=1}$, $s_k\ge0$, $k=1,2,...$, be given. Assume that $\{x_m\}_{m=0}^\infty$
is a sequence of non-negative
 numbers satisfying the inequalities
$$
x_0 \le C_1, \quad x_{m+1} \le x_m(1 - x_ms_{m+1}), \quad m = 0,1,2, \dots,\quad C_1>0 .
$$
Then we have for each $m$
$$
x_m \le \left(C_1^{-1}+\sum^m_{k=1} s_k\right)^{-1}.
$$
\end{Lemma}
\begin{proof} For completeness we present this simple proof here.  The proof is by induction on $m$. For $m = 0$ the statement
is true by assumption. We assume $x_m \le (C_1^{-1}+\sum^m_{k=1} s_k)^{-1}$ and prove that $x_{m+1} \le (C_1^{-1}+\sum^{m+1}_{k=1} s_k)^{-1}$. If $x_{m+1} = 0$ this statement is obvious. Assume therefore
that $x_{m+1} > 0$. Then we have 
$$
x_{m+1}^{-1} \ge x_m^{-1}(1 - x_ms_{m+1})^{-1} \ge x_m^{-1}(1 + x_ms_{m+1}) =
x_m^{-1} + s_{m+1} \ge C_1^{-1}+\sum^{m+1}_{k=1} s_k ,
$$
which implies $x_{m+1} \le (C_1^{-1}+\sum^{m+1}_{k=1} s_k)^{-1}$ . 
\end{proof}

We apply Lemma \ref{HL1} with $x_m:= a_mB_m^{-2}$. Then the inequality $\|f\| \le \|f\|_{A_1(\cD)}$ implies that we can take $C_1=1$. We set $s_k= b(2-b)t_k^2$, $k=1,2,...,$ and obtain from (\ref{aB}) and Lemma \ref{HL1}
\be\label{4.6}
a_mB_m^{-2} \le \left(1+ b(2-b)\sum_{k=1}^m t^2_k\right)^{-1}. 
\ee
Relations (\ref{4.3}) and (\ref{4.5}) imply
\be\label{4.7}
a_m \le a_{m-1}-b(2-b)y_mt_ma_{m-1}/B_{m-1} = a_{m-1}(1-b(2-b)t_my_m/B_{m-1}). 
\ee
We now need the following simple inequality: For any $x<1$ and any $a>0$ we have
\be\label{ineq}
(1-x)(1+x/a)^a \le 1.
\ee
Rewriting (\ref{4.4}) in the form
\be\label{4.9}
B_m = B_{m-1}(1+by_m/B_{m-1}) 
\ee
and using the inequality (\ref{ineq}) with $x=b(2-b)t_my_m/B_{m-1}$ and $a=(2-b)t_m$ we get from (\ref{4.7}) and (\ref{4.9}) that
\be\label{4.11}
a_mB_{m}^{(2-b)t_m} \le a_{m-1}B_{m-1}^{(2-b)t_m} .
\ee
Using monotonicity of $\{t_k\}$ (nonincreasing) and $\{B_k\}$ (increasing) we obtain
\be\label{4.11a}
B_{m-1}^{(2-b)t_m} = B_{m-1}^{(2-b)t_{m-1}}B_{m-1}^{(2-b)(t_m-t_{m-1})} \le B_{m-1}^{(2-b)t_{m-1}}B_0^{(2-b)(t_m-t_{m-1})}.
\ee
Finally,
\be\label{4.11b}
a_mB_{m}^{(2-b)t_m} \le   \|f\|^2\|f\|_{A_1(\cD)}^{(2-b)t_m}.
\ee

Combining (\ref{4.6}) and (\ref{4.11b}) we obtain
$$
a_m^{(2-b)t_m+2} \le \|f\|^4\|f\|_{A_1(\cD)}^{2(2-b)t_m} \left(1+b(2-b)\sum_{k=1}^mt_k^2\right)^{-(2-b)t_m},
$$
which completes the proof of Theorem \ref{RCT} with $\al_0 := \frac{(2-b)t_m}{(2-b)t_m+2}$.  The case $\al\le \al_0$ follows from Lemma 2.2 of \cite{VT196}.
\end{proof}

 This theorem implies that for any dictionary $\cD$, each $f$, and any realization $G^{\tau,b}_m(f,\cD)$ of the WGA($\cD,\tau,b$) we 
 have the following inequality
\be\label{P24}
\|f -G^{\tau,b}_m(f,\cD)\| \le \left(1+b(2-b) \sum^m_{k=1}t^2_k\right)^{-\al/2} \|f\|^{1-\al} \|f\|_{A_1(\cD)}^\al
\ee
provided $\al \le \frac{(2-b)t_m}{(2-b)t_m+2}$.

\begin{Corollary}\label{PC1} Assume $\tau :=\{t_k\}_{k=1}^\infty$ is a nonincreasing weakness sequence. For any Hilbert space $H$  we have
\be\label{P25}
\|f -G^{\tau,1}_m(f,\cD)\| \le \left(1+\sum^m_{k=1}t_k^2\right)^{-\al/2} \|f\|^{1-\al} \|f\|_{A_1(\cD)}^\al
\ee
provided $\al \le \frac{t_m}{t_m+2}$.
\end{Corollary}
 
 \begin{Remark}\label{tbR1} In the proof of Theorem \ref{RCT} we only used the greedy step 1) of the WGA($\cD,\tau,b$) for proving the inequality (\ref{4.5}). Therefore, for the Thresholding Weak Greedy Algorithm TWGA($\cD,\tau,b$), which is 
 the  WGA($\cD,\tau,b$) with the greedy step 1) replaced by the step
 $$
 |\<f_{m-1},\varphi_m\>| \ge t_ma_{m-1}/B_{m-1},
 $$
 Theorem \ref{RCT} holds as well. 
 \end{Remark}
 
 \section{Some  results on the rate of convergence of the DGA($\tau, b, \mu$)} 
\label{BS}

In this section we extend some results from the previous section to the case of a Banach space instead of a Hilbert space. We begin with some definitions. Let $X$ be a real Banach space with norm $\|\cdot\|$. As above we say that a set of elements (functions) $\cD$ from $X$ is a dictionary (symmetric dictionary) if each $g\in \cD$ has norm   one ($\|g\|= 1$), and $\csp \cD =X$. In addition we assume for convenience that
the dictionary is symmetric
$$
g\in \cD \quad \text{implies} \quad -g \in \cD.
$$
 It turns out that the geometrically defined class, namely, the closure of the convex hull of $\cD$, which we denote by $A_1(\cD)$, is a very natural class. For each $f\in X$ we associate the following norm
$$
\|f\|_{A_1(\cD)} := \inf \{M>0:\, f/M\in A_1(\cD)\}.
$$
Clearly, $\|f\|\le \|f\|_{A_1(\cD)}$. 

For a nonzero element $f\in X$ we denote by $F_f$ a norming (peak) functional for $f$: 
$$
\|F_f\| =1,\qquad F_f(f) =\|f\|.
$$
The existence of such a functional is guaranteed by Hahn-Banach theorem. 
Denote 
$$
r_\cD(f) := \sup_{F_f}\sup_{g\in \cD}F_f(g).
$$
We note that in general a norming functional $F_f$ is not unique. This is why we take $\sup_{F_f}$ over all norming functionals of $f$ in the definition of $r_\cD(f)$. It is known that in the case of uniformly smooth Banach spaces (our primary object here) the norming functional $F_f$ is unique. In such a case we do not need $\sup_{F_f}$ in the definition of $r_\cD(f)$.

 We consider here approximation in uniformly smooth Banach spaces. For a Banach space $X$ we define the modulus of smoothness
$$
\rho(u) := \rho(u,X):= \sup_{\|x\|=\|y\|=1}\left(\frac{1}{2}(\|x+uy\|+\|x-uy\|)-1\right).
$$
A uniformly smooth Banach space is  one with the property
$$
\lim_{u\to 0}\rho(u)/u =0.
$$
It is well known that in the case $X=L_p$, 
$1\le p < \infty$ we have
\be\label{1.3}
\rho(u,L_p) \le \begin{cases} u^p/p & \text{if}\quad 1\le p\le 2 ,\\
(p-1)u^2/2 & \text{if}\quad 2\le p<\infty. \end{cases} 
\ee

  We now give a definition of the DGA$(\tau,b,\mu)$, $\tau =\{t_k\}_{k=1}^\infty$, $t_k \in (0,1]$ introduced in \cite{VT111} (see also \cite{VTbook}, Ch.6).  
  
{\bf Dual Greedy Algorithm with parameters $(\tau,b,\mu)$ (DGA$(\tau,b,\mu)$).}
Let $X$ be a uniformly smooth Banach space with the modulus of smoothness $\rho(u)$ and let $\mu(u)$ be a continuous majorant of $\rho(u)$: $\rho(u)\le\mu(u)$, $u\in[0,\infty)$ with the property $\lim_{u\to +0}\mu(u)/u =0$. For a sequence $\tau =\{t_k\}_{k=1}^\infty$, $t_k \in (0,1]$ and a parameter $b\in (0,1]$ we define sequences
$\{f_m\}_{m=0}^\infty$, $\{\ff_m\}_{m=1}^\infty$, $\{c_m\}_{m=1}^\infty$, and $\{G_m\}_{m=0}^\infty$ inductively. Let $f_0:=f$ and $G_0:=0$. If for $m\ge 1$ $f_{m-1}=0$ then we set $f_j=0$ for $j\ge m$ and stop. If $f_{m-1}\neq 0$ then we conduct the following three steps:

1) take any $\ff_m \in \cD$ such that  (in the case $t_m=1$ we assume that  such a maximizer exists)
\be\label{3.1}
F_{f_{m-1}}(\ff_m) \ge t_mr_\cD(f_{m-1}); 
\ee

2) choose $c_m>0$ from the equation
\be\label{3.2}
\|f_{m-1}\|\mu(c_m/\|f_{m-1}\|) = \frac{t_mb}{2}c_mr_\cD(f_{m-1}); 
\ee

3) define
\be\label{3.3}
f_m:=f_{m-1}-c_m\ff_m,\qquad G_m:=G_m^{\tau,b,\mu}:= G_{m-1}+c_m\ff_m. 
\ee

Along with the algorithm DGA$(\tau,b,\mu)$ we consider a slight modification of it, when at step 2) we find 
$c_m$ from the  equation (see \cite{VT111}, Remark 3.1)
\be\label{3.2a}
\|f_{m-1}\|\mu(c_m/\|f_{m-1}\|) = \frac{b}{2}c_mF_{f_{m-1}}(\ff_m). 
\ee
We denote this modification by DGA$(\tau,b,\mu)^*$.

\begin{Remark}\label{BR1A} We now discuss when we can run the algorithm DGA$(\tau,b,\mu)$ for any $f\in X$. We begin with the step (\ref{3.1}). It is clear from the definition of the $r_\cD(f_{m-1})$ that in the case $t_m\in [0,1)$ we can always find 
a $\ff_m \in \cD$ satisfying (\ref{3.1}). In the case $t_m=1$ we need to impose an additional assumption that such $\ff_m \in \cD$ exists. We proceed to the step (\ref{3.2}). Consider the function $s(u):= \mu(u)/u$ defined on $(0,+\infty)$. Our assumptions guarantee that $s$ is a continuous function with the properties
$$
s(2) \ge \rho(2)/2 \ge 1/2 \quad \text{and} \quad \lim_{u\to +0}s(u) =0.
$$
On the other hand 
$$
\frac{t_mb}{2}r_\cD(f_{m-1}) \le 1/2.
$$
Therefore a solution of (\ref{3.2}) exists. 
\end{Remark}

We proceed to studying the rate of convergence of the DGA$(\tau,b,\mu)$ in the uniformly smooth Banach spaces with the power type majorant of the modulus of smoothness: $\rho(u)\le \mu(u)= \ga u^q$, $1<q\le 2$. The following Theorem \ref{T3.1} is from \cite{VT111} (see also \cite{VTbook}, p.372).

\begin{Theorem}[{\cite{VT111}}]\label{T3.1} Let $\tau :=\{t_k\}_{k=1}^\infty$ be a nonincreasing weakness sequence $1\ge t_1\ge t_2 \dots >0$ and $b\in (0,1)$. Assume that $X$ has a modulus of smoothness $\rho(u)\le \ga u^q$, $q\in (1,2]$. Denote $\mu(u) = \ga u^q$. Then for any dictionary $\cD$ and any $f\in A_1(\cD)$ the rate of convergence of the DGA$(\tau,b,\mu)$ is given by 
$$
\|f_m\|\le C(b,\ga,q)\left(1+\sum_{k=1}^mt_k^p\right)^{-\frac{t_m(1-b)}{p(1+t_m(1-b))}}, \quad p:= \frac{q}{q-1}.
$$
\end{Theorem}

\begin{Remark}\label{BR1} It is pointed out in \cite{VT111}, Remark 3.2, that Theorem \ref{T3.1} holds 
for the algorithm DGA$(\tau,b,\mu)^*$ as well.
\end{Remark}

Theorem \ref{T3.1} is an analog of Theorem \ref{T4.2}. We now prove an analog of Theorem \ref{RCT}. 

We extend Theorem \ref{T3.1} to provide a bound for  
\be\label{B1}
 \ga_m^{\tau,b,\mu}(\al,X):= \sup_{\cD}\sup_{f\in A_1(\cD), f\neq 0}\sup_{G^{\tau,b,\mu}_m(f,\cD)}\frac{\|f -G^{\tau,b,\mu}_m(f,\cD)\|}{\|f\|^{1-\al} \|f\|_{A_1(\cD)}^\al}  .
\ee
The corresponding characteristic for the algorithm DGA$(\tau,b,\mu)^*$ is denoted by $\ga_m^{\tau,b,\mu}(\al,X)^*$. 
Note that in the case $\tau =\{t_k\}$, $t_k=t$, $k=1,2,\dots$, Theorem~\ref{BT1} is proved in \cite{VT196}. 

\begin{Theorem}\label{BT1}  Let $\tau :=\{t_k\}_{k=1}^\infty$ be a nonincreasing weakness sequence $1\ge t_1\ge t_2 \dots >0$. For any Banach space $X$ with modulus of smoothness $\rho(u,X) \le \ga u^q$, $1<q\le 2$,  $p:= \frac{q}{q-1}$, we have
\be\label{B2}
\ga_m^{\tau,b,\mu}(\al,X) \le \left(1+c\sum^m_{k=1}t^p_k\right)^{-\al/p}, \quad  c := (1-b)\left(\frac{b}{2\ga}\right)^{\frac{1}{q-1}}, 
\ee
provided $\al \le \frac{t_m(1-b)}{1+t_m(1-b)}$. The same inequality holds for the $\ga_m^{\tau,b,\mu}(\al,X)^*$.
\end{Theorem}

\begin{proof} The proof is identical for both characteristics $\ga_m^{\tau,b,\mu}(\al,X)$ and $\ga_m^{\tau,b,\mu}(\al,X)^*$. We carry it out for the $\ga_m^{\tau,b,\mu}(\al,X)$. From the definition of the modulus of smoothness we have
\be\label{2.11}
\|f_{n-1}-c_n\ff_n\|+\|f_{n-1}+c_n\ff_n\| \le 2\|f_{n-1}\|(1+\rho(c_n/\|f_{n-1}\|)). 
\ee
Using the definition of $\ff_n$:
\be\label{2.12}
F_{f_{n-1}}(\ff_n) \ge t_nr_\cD(f_{n-1}) 
\ee
we get
\be\label{2.13}
\|f_{n-1}+c_n\ff_n\|\ge  F_{f_{n-1}}(f_{n-1}+c_n\ff_n) 
\ee
$$
= \|f_{n-1}\| +c_n F_{f_{n-1}}(\ff_n) \ge \|f_{n-1}\| +c_nt_nr_\cD(f_{n-1}). 
$$
Combining (\ref{2.11}) and (\ref{2.13}) we obtain
\be\label{2.14}
\|f_n\| = \|f_{n-1}-c_n\ff_n\| \le \|f_{n-1}\|(1+2\rho(c_n/\|f_{n-1}\|)) -c_nt_nr_\cD(f_{n-1}). 
\ee
Using the choice of $c_m$ we get from here
\be\label{3.9}
\|f_m\|\le \|f_{m-1}\| -t_m(1-b)c_mr_\cD(f_{m-1}). 
\ee
Thus, we need to estimate from below $c_mr_\cD(f_{m-1})$. It is clear that 
\be\label{3.10}
\|f_{m-1}\|_{A_1(\cD)} = \left\|f-\sum_{j=1}^{m-1}c_j\ff_j\right\|_{A_1(\cD)} \le \|f\|_{A_1(\cD)} +\sum_{j=1}^{m-1}c_j. 
\ee
Denote $B_n:= \|f\|_{A_1(\cD)}+\sum_{j=1}^nc_j$. Then by (\ref{3.10}) we have
$$
\|f_{m-1}\|_{A_1(\cD)} \le B_{m-1}.
$$
Next, by a well known relation (see, for instance, \cite{VTbook}, p.343, Lemma 6.10) we obtain
$$
r_\cD(f_{m-1}) = \sup_{g\in \cD} F_{f_{m-1}}(g) = \sup_{\ff \in A_1(\cD)} F_{f_{m-1}}(\ff) 
$$
\be\label{3.11}
\ge \|f_{m-1}\|_{A_1(\cD)}^{-1}F_{f_{m-1}}(f_{m-1}) \ge \|f_{m-1}\|/B_{m-1}.
\ee
Substituting (\ref{3.11}) into (\ref{3.9}) we get
\be\label{3.12}
\|f_m\| \le \|f_{m-1}\|(1-t_m(1-b)c_m/B_{m-1}). 
\ee
From the definition of $B_m$ we find
$$
B_m = B_{m-1} +c_m = B_{m-1}(1+c_m/B_{m-1}).
$$
Using the inequality
$$
(1+x)^\alpha \le 1+\alpha x, \quad 0\le \alpha\le 1, \quad x\ge 0,
$$
we obtain
\be\label{3.13}
B_m^{t_m(1-b)}  \le  B_{m-1}^{t_m(1-b)}(1+t_m(1-b)c_m/B_{m-1}). 
\ee
Multiplying (\ref{3.12}) and (\ref{3.13})  we get
\be\label{3.14}
\|f_m\|B_m^{t_m(1-b)}  \le \|f_{m-1}\| B_{m-1}^{t_m(1-b)}  . 
\ee
Using monotonicity of $\{t_k\}$ (nonincreasing) and $\{B_k\}$ (increasing), we obtain
\be\label{3.14a}
 B_{m-1}^{t_m(1-b)}\le  B_{m-1}^{t_{m-1}(1-b)}B_0^{(t_m-t_{m-1})(1-b)}.
\ee
Finally, we obtain
\be\label{3.14b}
\|f_m\|B_m^{t_m(1-b)} \le   \|f\| \|f\|_{A_1(\cD)}^{t_m(1-b)}.
\ee

The function $\mu(u)/u = \ga u^{q-1}$ is increasing on $[0,\infty)$. Therefore the $c_m$ from (\ref{3.2}) is greater than or equal to $c_m'$ from (see (\ref{3.11}))
\be\label{3.15}
\ga \|f_{m-1}\|(c_m'/\|f_{m-1}\|)^q = \frac{t_mb}{2}c_m'\|f_{m-1}\|/B_{m-1}, 
\ee
\be\label{3.16}
c_m' = \left(\frac{t_mb}{2\ga}\right)^{\frac{1}{q-1}}\frac{\|f_{m-1}\|^{\frac{q}{q-1}}}{B_{m-1}^{\frac{1}{q-1}}}. 
\ee
Using notations
$$
p:=\frac{q}{q-1},\qquad c := (1-b)\left(\frac{b}{2\ga}\right)^{\frac{1}{q-1}},
$$
we get from (\ref{3.9}), (\ref{3.11}), (\ref{3.16})
\be\label{3.17}
\|f_m\|\le \|f_{m-1}\| \left(1-ct^p_m\frac{\|f_{m-1}\|^p}{B_{m-1}^p}\right). 
\ee
Noting that $B_m\ge B_{m-1}$ we derive from (\ref{3.17}) that
\be\label{3.18}
(\|f_m\|/B_m)^p \le (\|f_{m-1}\|/B_{m-1})^p (1-ct^p_m(\|f_{m-1}\|/B_{m-1})^p). 
\ee
Taking into account that $\|f\|\le \|f\|_{A_1(\cD)}$ we obtain from (\ref{3.18}) by Lemma \ref{HL1} with $C_1=1$, $s_k=ct_k^p$, $k=1, 2,...,$
\be\label{3.19}
(\|f_m\|/B_m)^p \le \left(1+c \sum^m_{k=1}t_k^p\right)^{-1}. 
\ee
Combining (\ref{3.14b}) and (\ref{3.19}) we obtain
$$
\|f_m\|\le \|f\|^{1-\al_0}\|f\|_{A_1(\cD)}^{\al_0}\left(1+c\sum^m_{k=1} t^p_k\right)^{-\al_0/p}
$$
with
$$
  p:= \frac{q}{q-1},\quad \al_0 := \frac{t_m(1-b)}{1+t_m(1-b)}.
$$
This completes the proof of Theorem \ref{BT1} for $\al=\al_0$. The case $\al <\al_0$ follows from 
the case $\al=\al_0$ and the corresponding analog of Lemma 2.2 from \cite{VT196}.
\end{proof}

\section{Convergence of the Weak Thresholding Greedy Algorithm for $A_1(\cD)$}
\label{T}

In this section we discuss the new setting  of  
studying convergence of an algorithm for a subset $A_1(\cD)$ of $X$ instead of the whole space $X$. 
 We are interested in the following general convergence problems.
 
 {\bf Problem 1 (individual dictionary).} For a given dictionary $\cD$ find necessary and sufficient conditions on the weakness sequence $\tau$, which guarantee convergence of a specific weak greedy algorithm GA($\tau$) with respect to $\cD$ for each element $f\in A_1(\cD)$.
 
 {\bf Problem 2 (all dictionaries).} Find necessary and sufficient conditions on the weakness sequence $\tau$, which guarantee convergence of a specific weak greedy algorithm GA($\tau$) with respect to any dictionary $\cD$  for each element $f\in A_1(\cD)$.

We discuss here {\bf Problem 1} on a very special example, when the dictionary $\cD$ is a basis in a Banach space $X$ or 
an orthonormal basis in a Hilbert space $H$. For further discussion see Section \ref{D}. 

We give a definition of the Weak Thresholding Greedy Algorithm in a general Banach space. 
 
 We begin with the definition of the Thresholding Greedy Algorithm (TGA), which is the Weak Thresholding Greedy Algorithm
 with the weakness parameter $t=1$.  Let a Banach space $X$, with a basis $\Psi =\{\psi_k\}_{k=1}^\infty$,    be given. We assume that $\|\psi_k\|\ge C>0$, $k=1,2,\dots$,  and consider the following theoretical greedy algorithm. For a given element $f\in X$ we consider the expansion
\begin{equation}\label{BIn1}
f = \sum_{k=1}^\infty c_k(f,\Psi)\psi_k. 
\end{equation}
For an element $f\in X$ we say that a permutation $\rho$ of the positive integers   is  decreasing  if 
\begin{equation}\label{BIn2}
|c_{k_1}(f,\Psi) |\ge |c_{k_2}(f,\Psi) | \ge \dots ,  
\end{equation}
 where $\rho(j)=k_j$, $j=1,2,\dots$, and write $\rho \in D(f)$.
If the   inequalities are strict in (\ref{BIn2}), then $D(f)$ consists of only one permutation. We define the $m$th greedy approximant of $f$, with regard to the basis $\Psi$ corresponding to a permutation $\rho \in D(f)$, by the formula
$$
G_m(f):=G_m(f,\Psi)  :=G_m(f,\Psi,\rho) := \sum_{j=1}^m c_{k_j}(f,\Psi)\psi_{k_j}.
$$

We now give the definition of the Weak Thresholding Greedy Algorithm
 with weakness sequence $\tau$ (WTGA($\tau$)), which was introduced and studied in \cite{KamTem} (see also \cite{VTbook}, pp. 42-43). At the first iteration we find $k_1$ such that 
 $$
 |c_{k_1}(f,\Psi) | \ge t_1\max_k|c_{k}(f,\Psi) |
 $$
 and define
 $$
 G_1^\tau(f):=G_1(f,\Psi,\tau) := c_{k_1}(f,\Psi)\psi_{k_1},\quad f_1^\tau := f-G_1^\tau(f).
 $$
 At the $m$th iteration ($m>1$) we find $k_m$ such that 
 $$
 |c_{k_m}(f_{m-1},\Psi) | \ge t_m \max_k|c_{k}(f_{m-1},\Psi) |
 $$
 and define
 $$
 G_m^\tau(f):=G_m(f,\Psi,\tau) :=G_{m-1}^\tau(f) +c_{k_m}(f_{m-1},\Psi)\psi_{k_m},\quad f_m^\tau := f-G_m^\tau(f).
 $$
Clearly, $c_{k_m}(f_{m-1},\Psi) = c_{k_m}(f,\Psi)$.

   Consider the case, when $X$ is a Hilbert space $H$ and the dictionary $\cD$ is an orthonormal basis of $H$. 
 It is easy to see that in this case the following three greedy algorithms coincide: the Weak Thresholding Greedy Algorithm
 with weakness sequence $\tau$ (WTGA($\tau$)), defined above, the Weak Greedy Algorithm (WGA($\tau$)) (see, \cite{VTbook}, p.80), and the Weak Orthogonal Greedy Algorithm (WOGA($\tau$)) (see, \cite{VTbook}, p.81). 
 
 \begin{Theorem}\label{TT1} Let $\Psi$ be an orthonormal basis for a Hilbert space $H$. The following condition 
 \be\label{T1}
 \sum_{k=1}^\infty t_k =\infty
 \ee
 is the necessary and sufficient condition on the weakness sequence $\tau$ for convergence of the WTGA($\tau$)  with respect to $\Psi$ for each element $f\in A_1(\Psi)$.
  \end{Theorem}
  \begin{Remark}\label{TR1a} We formulated Theorem \ref{TT1} for the algorithm WTGA($\tau$). This theorem holds for the 
  algorithms WGA($\tau$) and WOGA($\tau$) as well.
  \end{Remark}
  \begin{proof} Denote for brevity 
  $$
  c_k(f) := c_k(f,\Psi)\quad  \text{and}\quad  \bc (f):= \{c_k(f)\}_{k=1}^\infty.
  $$
  For a sequence $\bx=\{x_k\}_{k=1}^\infty$ use the standard notations ($1\le p <\infty$)
  $$
  \|\bx\|_p  := \left(\sum_{k=1}^\infty |x_k|^p\right)^{1/p},\quad \|\bx\|_\infty := \sup_k|x_k|.
  $$
  It is easy to understand that in our case relation $f\in A_1(\Psi)$ means that $\|\bc (f)\|_1 \le 1$. 
  
  {\bf Sufficiency.} Suppose that the condition (\ref{T1}) is satisfied. First, we prove that 
  \be\label{T2}
  \lim_{m\to\infty} \|\bc(f^\tau_m)\|_\infty =0.
  \ee
  We prove it by contradiction. It is clear that the sequence $\{\|\bc(f^\tau_m)\|_\infty\}$ is monotone decreasing. Therefore,
  assuming that (\ref{T2}) does not hold, we obtain that there exists an $a>0$ such that $\|\bc(f^\tau_m)\|_\infty \ge a$ for all $m$.
  Then
  $$
  \sum_{k=1}^\infty |c_k(f)| \ge \sum_{j=1}^\infty |c_{k_j}(f)| \ge a \sum_{j=1}^\infty t_j =\infty.
  $$
  We obtain the contradiction, which completes the proof of (\ref{T2}). Next,
  $$
  \|f_m^\tau\|_2^2 = \|\bc(f_m^\tau)\|_2^2 \le  \|\bc(f_m^\tau)\|_\infty \|\bc(f_m^\tau)\|_1.
  $$
  Our assumption $f\in A_1(\Psi)$ guarantees that $\|\bc(f_m^\tau)\|_1 \le \|\bc(f)\|_1 \le 1$. 
  Therefore, (\ref{T2}) implies that $ \|f_m^\tau\|_2 \to 0$ with $m\to \infty.$
  
  {\bf Necessity.} Suppose that condition (\ref{T1}) does not hold, which means that 
   \be\label{T3}
 \sum_{k=1}^\infty t_k < \infty.
 \ee
 Consider the function
 $$
 f_0 := \psi_1+\sum_{k=2}^\infty t_{k-1} \psi_k.
 $$
 Then we have  $f^* := f_0\left(1+\sum_{k=1}^\infty t_k\right)^{-1} \in A_1(\Psi)$. 
 For the function $f_0$ one of the realizations of the WTGA($\tau$) is the following
 $$
 G_m^\tau(f_0) = \sum_{k=2}^{m+1} t_{k-1}\psi_k.
 $$
 Clearly, this realization does not converge to $f_0$. 
  
\end{proof}

\begin{Remark}\label{TR1} Theorem \ref{TT1} holds for a Banach space $X$ and any basis $\Psi$
satisfying the conditions $0<C_1\le \|\psi_k\| \le C_2<\infty$. 
\end{Remark} 
\begin{proof} The proof of necessity repeats the corresponding one from Theorem \ref{TT1}. 
The proof of (\ref{T2}) in the sufficiency part goes for any basis. The relation (\ref{T2}) implies 
that for any $N$ there exists $m(N)$ such that $c_k(f_{m(N)})=0$ provided $k\le N$. Then,
\be\label{T4}
\|f_{m(N)}\| \le \sum_{k=N+1}^\infty |c_k(f)|\|\psi_k\| \le C_2 \sum_{k=N+1}^\infty |c_k(f)|.
\ee
It remains to note that the right hand side of (\ref{T4}) goes to zero when $N$ goes to $\infty$. 

\end{proof}
 
 \section{Applications of weak greedy algorithms for remote consecutive projections}
 \label{P}
 
 We refer the reader to the survey \cite{BK} for a detailed discussion of known results in the theory of consecutive projections and relations between the theory of greedy approximation and the theory of consecutive projections. 
 We begin with the definition of the Weak Remote Projections Algorithm. Let $H$ be a Hilbert space with the unit sphere $S(H):=\{x\in H:\|x\|=1\}$ and 
 $\cL:= \{L\}$ be a collection of closed subspaces of $H$. 
 
 {\bf Weak Remote Projections Algorithm (WRPA($\cL,\tau$)).} We start with $x_0\in H$. The first iteration of this algorithm 
 provides $x_1:= \Pr_{L_1}(x_0)$, which is the projection of $x_0$ onto a subspace $L_1\in \cL$ satisfying the condition
  (in the case $t_1=1$ we assume that  such a maximizer exists)
 \be\label{P1}
 \dist(x_0,L_1) \ge t_1\sup_{L\in \cL}  \dist(x_0,L).
 \ee
At the $n$th iteration ($n>1$) we build $x_n := \Pr_{L_n}(x_{n-1})$, where a subspace $L_n\in \cL$ satisfies the condition
(in the case $t_n=1$ we assume that  such a maximizer exists)
 \be\label{P2}
 \dist(x_{n-1},L_n) \ge t_n\sup_{L\in \cL}  \dist(x_{n-1},L).
 \ee

 In our further discussion we assume that $\cap_{L\in \cL} L = \{0\}$. With each $L\in\cL$ we associate the system 
 $\cD_L := L^\perp\cap S(H)$ and define the dictionary $\cD(\cL) := \cup_{L\in \cL}\cD_L$. Note that it is easy to see that 
 the condition $\cap_{L\in \cL} L = \{0\}$ implies that the closure of $\sp\{\cD(\cL)\}$ coincides with $H$ and, therefore, 
 $\cD(\cL)$ is a dictionary. Note that 
 \be\label{P2a}
 \Pr_L(x) = x-\Pr_{L^\perp}(x).
 \ee
 It follows from the definition of $\cD_L$ that there exists $g\in \cD_L$ such that $\Pr_{L^\perp}(x)= \<x,g\>g$. In other words,
 for any $L\in\cL$ there exists $g\in \cD_L$ such that
 \be\label{P3}
 \Pr_L(x) = x-\<x,g\>g.
 \ee
 
  Relation (\ref{P3}) establishes a connection between the WRPA($\cL,\tau$) and weak greedy algorithms. 
  
\begin{Remark}\label{PR1} Relation (\ref{P3}) and the obvious relation
\be\label{P7}
\dist(x,L) = \|P_{L^\perp}(x)\| = |\<x,g\>|
\ee
show that the WRPA($\cL,\tau$) is a realization of the WGA($\cD(\cL),\tau$).
\end{Remark}

 \subsection{Some known results on greedy algorithms in Hilbert spaces}
 \label{GA}
 
   The following theorem was   proved in \cite{VT75} (see also, \cite{VTbook}, p.85).
\begin{Theorem}[{\cite{VT75}}]\label{PT1} Assume
\be\label{P8}
\sum_{k=1}^\infty \frac{t_k}{k} = \infty.  
\ee
Then for any dictionary $\cD$ and any $f\in H$ we have for the WGA($\cD,\tau$)
$$
\lim_{m\to \infty}\|f-G_m^\tau(f,\cD)\| =0.
$$
\end{Theorem}

 Theorem \ref{PT1} and Remark \ref{PR1} imply the following result.
 
 \begin{Theorem}\label{PT2} Assume that the weakness sequence $\tau$ satisfies (\ref{P8}). 
Then for any collection $\cL$ satisfying the condition $\cap_{L\in \cL} L = \{0\}$  we have for the WRPA($\cL,\tau$)
$$
\lim_{n\to \infty}\|x_n\| =0.
$$
\end{Theorem}

 The following theorem, which is stronger than Theorem \ref{PT1}, was proved in \cite{VT81}.
\begin{Theorem}[{\cite{VT81}}]\label{PT3} Assume
\be\label{P9}
\sum_{s=0}^\infty \left(2^{-s}\sum_{k=2^s}^{2^{s+1}-1}t_k^2\right)^{1/2} = \infty.  
\ee
Then for any dictionary $\cD$ and any $f\in H$ we have for the WGA($\cD,\tau$)
$$
\lim_{m\to \infty}\|f-G_m^\tau(f,\cD)\| =0.
$$
\end{Theorem}

 Theorem \ref{PT3} and Remark \ref{PR1} imply the following result.
 
 \begin{Theorem}\label{PT4} Assume that the weakness sequence $\tau$ satisfies (\ref{P9}). 
Then for any collection $\cL$ satisfying the condition $\cap_{L\in \cL} L = \{0\}$  we have for the WRPA($\cL,\tau$)
$$
\lim_{n\to \infty}\|x_n\| =0.
$$
\end{Theorem}

The following special case of weakness sequences is of interest. The reader can find some results on this special case in \cite{VT81}. Let 
$$
\cN :=\{n_k\}_{k=1}^\infty, \quad n_k\in \bbN, \quad n_k <n_{k+1}, \quad k=1,2,\dots. 
$$
For $t\in (0,1]$ denote 
$\tau(\cN,t) := \{t_n\}_{n=1}^\infty$ the weakness sequence with the property:
\be\label{P9a}
t_n=t \quad \text{for}\quad n\in \cN,\qquad t_n=0 \quad \text{for}\quad n\notin \cN.
\ee

The following results is obtained in \cite{VT81} (see Lemma 3.1 and Remark 3.1 there).

\begin{Theorem}[{\cite{VT81}}]\label{PT5} For a given subsequence $\cN$ assume
\be\label{P9b}
\sum_{k=1}^\infty  \frac{(n_{k+1}-n_k)^{1/2}}{n_k} = \infty, \qquad n_{k+1}/n_k \le C_0 
\ee
with a constant $C_0$.
Then for any dictionary $\cD$ and any $f\in H$ we have for the WGA($\cD,\tau(\cN,t)$)
$$
\lim_{m\to \infty}\|f-G_m^\tau(f,\cD)\| =0.
$$
\end{Theorem}


We now prove a generalization of Theorem \ref{PT5}. The proof is very similar to the proof of Theorem \ref{PT5} from \cite{VT81}. However, for completeness we present it here. 

\begin{Theorem}\label{PT5a} Suppose that a subsequence $\cN$ and a weakness sequence $\tau$ satisfy the conditions
\be\label{P9c}
\sum_{k=1}^\infty  \frac{(n_{k+1}-n_k)^{1/2}t_{n_k}}{n_k} = \infty, \qquad n_{k+1}/n_k \le C_0, \quad k=1,2,\dots, 
\ee
with a constant $C_0$.
Then for any dictionary $\cD$ and any $f\in H$ we have for the WGA($\cD,\tau$)
$$
\lim_{m\to \infty}\|f-G_m^\tau(f,\cD)\| =0.
$$
\end{Theorem}
\begin{proof} We will prove that
\be\label{I}
\sum_{k=1}^\infty \frac{(n_{k+1}-n_k)^{1/2}t_{n_k}}{n_k} = \infty
\ee
implies the following property
\be\label{II}
\forall \{a_j\} \in \ell_2,\quad a_j\ge 0, j=1,2,\dots,\quad\liminf_{k\to \infty}\frac{a_{n_k}}{t_{n_k}}\sum_{j=1}^{n_k}a_j =0.
\ee
Then we apply the criterion for convergence of the WGA($\cD,\tau$) from \cite{VT82} (see Theorem \ref{DT1} below) and complete the proof. So, we return back to the proof of (\ref{I}) $\Rightarrow$ (\ref{II}). 
We prove the following a little stronger statement than (\ref{II}): $\forall \{a_j\} \in \ell_2$, $ a_j\ge 0$, $j=1,2,\dots$,
\be\label{3.2.1}
\sum_{k=1}^\infty\frac{(n_{k+1}-n_k)^{1/2}t_{n_k}}{n_k}\frac{a_{n_k}}{t_{n_k}}\sum_{j=1}^{n_k}a_j <\infty.
\ee
It is known (see \cite{Z}, Ch.1, S.9) that $\{a_j\} \in \ell_2$ implies that
\be\label{3.3.1}
\{b_n\}_{n=1}^\infty \in \ell_2 \quad \text{with}\quad b_n := \frac{1}{n}\sum_{j=1}^na_j.
\ee
We observe that
\be\label{3.4}
\sum_{k=1}^\infty (n_{k+1}-n_k)b_{n_k}^2 <\infty.
\ee
Indeed, for any $m>n$ we have
$$
mb_m\ge nb_n
$$
 and for $n_k<m<n_{k+1}$ we have  
$$
b_{n_k} \le \frac{n_{k+1}}{n_k} b_m \le C_0b_m
$$
with a constant $C_0$ independent of $k$ and $m$. Therefore,
\be\label{3.5}
(n_{k+1}-n_k)b_{n_k}^2 \le C_0^2 \sum_{m=n_k}^{n_{k+1}-1}b_m^2 .
\ee
Combining (\ref{3.3.1}) and (\ref{3.5}) we obtain (\ref{3.4}).

We return to (\ref{3.2.1})
$$
\sum_{k=1}^\infty\frac{(n_{k+1}-n_k)^{1/2}t_{n_k}}{n_k}\frac{a_{n_k}}{t_{n_k}}\sum_{j=1}^{n_k}a_j =
\sum_{k=1}^\infty(n_{k+1}-n_k)^{1/2} a_{n_k}b_{n_k} \le
$$
$$
\left(\sum_{k=1}^\infty a_{n_k}^2\right)^{1/2}\left(\sum_{k=1}^\infty(n_{k+1}-n_k)b_{n_k}^2\right)^{1/2} <\infty.
$$

\end{proof}

\begin{Remark}\label{PR2} In the case of subsequence $n_k=k$, $k=1,2,\dots$, Theorem \ref{PT5a} gives Theorem \ref{PT1}
and in the case $\tau = \tau(\cN,t)$ it gives Theorem \ref{PT5}.
\end{Remark}

Theorem \ref{PT5a} and Remark \ref{PR1} imply the following result.
 
 \begin{Theorem}\label{PT6} Assume that a subsequence $\cN$   and a weakness sequence $\tau$ satisfy the conditions (\ref{P9c}). 
Then for any collection $\cL$ satisfying the condition $\cap_{L\in \cL} L = \{0\}$  we have for the WRPA($\cL,\tau$)
$$
\lim_{n\to \infty}\|x_n\| =0.
$$
\end{Theorem}

Here is a corollary of Theorem \ref{InT1}.

 \begin{Theorem}\label{PT7} Assume $\tau :=\{t_k\}_{k=1}^\infty$ is a nonincreasing weakness sequence. For any Hilbert space $H$ for any collection $\cL$ satisfying the condition $\cap_{L\in \cL} L = \{0\}$ for each $f\in H$ we have for the WRPA($\cL,\tau$)   
\be\label{P26}
\|f_m\| \le \left(1+\sum^m_{k=1} t^2_k\right)^{-\al/2} \|f\|^{1-\al} \|f\|_{A_1(\cD(\cL))}^\al
\ee
provided $\al \le \frac{t_m}{t_m+2}$.
\end{Theorem}

\subsection{ Projections algorithms versus greedy algorithms}
\label{PvG}

We have illustrated above the well known phenomenon (see \cite{BK} and references therein) that consecutive projections algorithms are related to greedy algorithms. We now discuss this important phenomenon in detail. 
Let as above $\cL$ be a collection of closed subspaces of a Hilbert space $H$. Associate with $\cL$ the collection $\cL^\perp:= \{L^\perp: L\in\cL\}$ and the dictionary $\cD(\cL)= (\cup_{L\in\cL} L^\perp)\cap S(H)$. We begin with the greedy step (\ref{P2}). For convenience, we use standard notations of the greedy approximation theory.

 $\cL$-greedy step: Find $L_m \in \cL$ such that
 \be\label{P10}
   \dist(f_{m-1},L_m) \ge t_m\sup_{L\in \cL}  \dist(f_{m-1},L).
 \ee
 By the following obvious relations 
  \be\label{P11}
 \Pr_L(x) = x-\Pr_{L^\perp}(x),\qquad \dist(x,L)= \|\Pr_{L^\perp}(x)\|
 \ee
 we see that the greedy step (\ref{P10}) with respect to the collection $\cL$ is equivalent to the following greedy step with respect to the dictionary $\cD(\cL)$.
 
  $\cD(\cL)$-greedy step: Find $L_m^\perp \in \cL^\perp$ such that
 \be\label{P12}
    \|\Pr_{L_m^\perp}(f_{m-1})\| \ge t_m\sup_{L\in \cL}  \|\Pr_{L^\perp}(f_{m-1})\|.
 \ee
 Note that the above greedy step provides the following step, which is the greedy step with respect to $\cD(\cL)$: Find $\ff_m \in \cD(\cL)$ such that 
 \be\label{P13}
 |\<f_{m-1},\ff_m\>| \ge t_m \sup_{g\in \cD(\cL)} |\<f_{m-1},g\>|.
 \ee
 This means that the algorithm with the step (\ref{P12}) is a realization of the weak greedy algorithm with respect to $\cD(\cL)$. 
 For brevity, denote 
 $$
 P_m(f):=\<f_{m-1},\ff_m\>\ff_m .
 $$ 

There is a number of greedy algorithms, which have the greedy step (\ref{P13}) (in other words, expressed in terms of projections, (\ref{P12})). These algorithms differ in the approximation steps and have the same greedy step (with respect to a given dictionary $\cD$)
  \be\label{P13a}
 |\<f_{m-1},\ff_m\>| \ge t_m \sup_{g\in \cD} |\<f_{m-1},g\>|.
 \ee
We now recall some of them. The WGA($\cD,\tau$)) is defined above. 

The Weak Orthogonal Greedy Algorithm WOGA($\cD,\tau$)) with respect to a dictionary $\cD$ has the greedy step (\ref{P13a}) for $m=1,2,\dots$ ($f_0:=f$). Its approximation step is the following: ($G_m(f)$ is an approximant, $G_0(f)=0$, $f_m$ is a residual, $f_0:=f$) 
 \be\label{P14}
   G_m(f) := \Pr_{H_m}(f),\quad H_m := \sp(\ff_1,\dots,\ff_m),\quad f_m := f- G_m(f).
 \ee
 
 The Weak Greedy Algorithm with Free Relaxation has the greedy step (\ref{P13a})  for $m=1,2,\dots$ ($f_0:=f$). Its approximation step is the following: ($G_m(f)$ is an approximant $G_0(f)=0$, $f_m$ is a residual, $f_0:=f$) 
 \be\label{P15}
   G_m(f) := (1-w^*)G_{m-1}(f) +\la^* \ff_m,\quad f_m := f- G_m(f),
 \ee
 where
  \be\label{P16}
  \|f-((1-w^*)G_{m-1}(f) +\la^* \ff_m)\| = \min_{w,\la} \|f-((1-w)G_{m-1}(f) +\la \ff_m)\| .
  \ee
 In other words, 
 $$
 G_m(f) := \Pr_{Y_m}(f),\qquad Y_m := \sp(G_{m-1}(f),\ff_m).
 $$
 
 There are greedy algorithms, which are designed for approximation of elements from the special class $\conv(\cD)$ (closure of convex hull of $\cD$). We present two examples. 
 
 Relaxed Greedy Algorithm (RGA) similarly to the above algorithms has the greedy step (\ref{P13a}) but for a special weakness sequence $\tau =\{1\}$ (see \cite{VTbook}, p.82).  Its approximation step is the following: 
 $$
 G_0(f) :=0,\qquad G_1(f) := P_1(f),
 $$
 and for $m\ge 2$
  \be\label{P17}
   G_m(f) := (1-1/m)G_{m-1}(f) +(1/m)\ff_m,\quad f_m :=f-G_m(f).
  \ee
  
  We proceed to a thresholding-type algorithm (see \cite{VT94}).  Keeping in mind possible applications of this algorithm we  consider instead of $A_1(\cD)$  the closure of the convex hull of $\cD$, which we denote conv($\cD$) or  $A_1(\cD^+)$. Clearly, $A_1(\cD) = \conv(\cD^\pm)$, where $\cD^\pm := \{\pm g\,:\, g\in \cD\}$ is the symmetrized version of dictionary $\cD$. 
 Let $\e=\{\e_n\}_{n=1}^\infty $, $\e_n> 0$, $n=1,2,\dots$ . 

 {\bf Incremental Algorithm with schedule $\e$ (IA($\e$)).} 
Let $f\in \conv(\cD)$. Denote $f_0^{i,\e}:= f$ and $G_0^{i,\e} :=0$. Then, for each $m\ge 1$ we have the following inductive definition.

(1) $\ff_m^{i,\e} \in \cD$ is any element satisfying
$$
\<f_{m-1}^{i,\e},(\ff_m^{i,\e}-f)\> \ge -\e_m \|f_{m-1}^{i,\e}\|.
$$

(2) Define
$$
G_m^{i,\e}:= (1-1/m)G_{m-1}^{i,\e} +\ff_m^{i,\e}/m.
$$

(3) Let
$$
f_m^{i,\e} := f- G_m^{i,\e}.
$$ 

Let us demonstrate how the IA($\e$) algorithm works in the case of remote  projections. We call the corresponding algorithm the Remote Projections with schedule $\e$ (RP($\e$)). 
We start with $x_0$ and a collection $\cL = \{L\}$ of subspaces. Denote $f:= x_0$ and as above assume that $\cap_{L\in \cL} L = \{0\}$. With each $L\in\cL$ we associate the system 
 $\cD_L := L^\perp\cap S(H)$ and define the dictionary $\cD(\cL) := \cup_{L\in \cL}\cD_L$. Assume that $f \in \conv(\cD(\cL))$.
 At the first step the IA($\e$) looks for an element $\ff_1\in \cD(\cL)$ satisfying
 \be\label{PI1}
 \<f,\ff_1\> \ge \|f\|^2 - \e_1\|f\|.
 \ee
 The theory of greedy approximation guarantees existence of such element $\ff_1$. This step is equivalent to finding a subspace $L_1\in \cL$ such that 
 \be\label{PI2}
 \dist(x_0,L_1) \ge \|x_0\|^2 -\e_1\|x_0\|,\qquad x_1:=x_0 -\ff_1,
 \ee
 where $\ff_1\in \cD_{L_1}$ is such that $\Pr_{L_1}(x_0) = x_0 - \<x_0,\ff_1\>\ff_1$. 
 Then we iterate these steps in such a way that the $m$th step of the RP($\e$) is the following. Find a subspace $L_m\in \cL$ such that 
 $$
  \dist(x_{m-1},L_m) \ge \|x_{m-1}\|^2 -\e_m\|x_{m-1}\|,
 $$
 \be\label{PI3}
  x_m := x_{m-1} -(\ff_1+\dots+\ff_m)/m, 
 \ee
 where $\ff_m\in \cD_{L_m}$ is such that $\Pr_{L_m}(x_{m-1}) = x_{m-1} - \<x_{m-1},\ff_m\>\ff_m$. 

\section{Discussion}
\label{D}

We begin with some general comments on convergence of greedy algorithms. The reader can find a detailed discussion 
in the recent survey \cite{VT211}. For brevity we write GA for greedy algorithm.

{\bf Generic GA.} GA is an iterative process. Each iteration of it consists of two steps -- 
a greedy step and an approximation step. GA works in a Banach space $X$ with respect to a given 
system (dictionary) $\cD$ of elements. For a given element $f\in X$ it builds a sequence of elements $\ff_1,\ff_2,\dots$ of 
the system $\cD$, a sequence of approximations $G_1,G_2,\dots$, and a sequence of residuals $f_k:= f-G_k$, $k=1,2,\dots$. For convenience we set $G_0=0$ and $f_0=f$. Suppose we have performed $m-1$ iterations of GA. 
Then, at the $m$th iteration we apply the greedy step of our specific GA. The greedy step provides a new element 
$\ff_m\in \cD$ to be used for approximation. After that, using $\ff_m$ and the information from the previous $m-1$ iterations, we apply the approximation step of our specific GA and obtain the approximation $G_m$. 
Specifying the greedy steps and the approximation steps, we obtain different greedy algorithms. 
The reader can find the most basic examples of greedy and approximation steps in \cite{VT211}. 

{\bf Convergence.} We say that an algorithm GA converges in a Banach space $X$ if for any dictionary $\cD$ and each $f\in X$ we have for any realization of the algorithm GA the following property 
$$
\lim_{m\to \infty} \|f_m\|_X =0.
$$

In the above definition of convergence we ask for convergence of an algorithm GA for all elements $f$ of $X$. In this paper 
(see Section \ref{T}) we discuss convergence of a greedy algorithm not for all elements $f$ of $X$ but only for 
$f\in A_1(\cD)$. It seems to be an interesting and important problem. One would expect this kind of convergence under weaker conditions, for instance conditions on the weakness sequence, than the corresponding conditions for the above defined convergence. We formulate the corresponding definition.

{\bf $A_1(\cD)$-convergence.} We say that an algorithm GA $A_1(\cD)$-converges in a Banach space $X$ if for a given dictionary $\cD$ and each $f\in A_1(\cD)$ we have for any realization of the algorithm GA the following property 
$$
\lim_{m\to \infty} \|f_m\|_X =0.
$$

It is well known that the problem of finding a criterion for convergence of a specific greedy algorithm (for instance, in terms of 
the weakness sequence) is a difficult problem and it is only solved in a few cases. The reader can find related open problems, for instance, in \cite{VT211}. We now formulate the criterion for convergence of the WGA($\tau$) from \cite{VT82}, which covers all weakness sequences~$\tau$.

\begin{Theorem}[{\cite{VT82}}]\label{DT1} The following condition on the weakness sequences $\tau$ 
\be\label{D1}
\forall \{a_j\} \in \ell_2,\quad a_j\ge 0, j=1,2,\dots,\quad\liminf_{n\to \infty}\frac{a_{n}}{t_{n}}\sum_{j=1}^{n}a_j =0
\ee
is the necessary and sufficient condition for convergence of the WGA($\tau$).
\end{Theorem} 

We presented above some sufficient conditions on the weakness sequence $\tau$, which guarantee convergence of the WGA($\tau$) for all $\cD$ and each $f\in H$. There are some criteria (see \cite{VT81}) for convergence of the WGA($\tau$) under extra assumptions on $\tau$. We formulate one of those from \cite{VT81}. 

\begin{Theorem}[{\cite{VT81}}]\label{DT2} Assume that the weakness sequences $\tau =\{t_k\}_{k=1}^\infty$ are monotone $t_1\ge t_2 \ge \dots$. Then, the following condition on the weakness sequences~$\tau$ 
\be\label{D2}
\sum_{k=1}^\infty \frac{t_k}{k} =\infty
\ee
is the necessary and sufficient condition for convergence of the WGA($\tau$).
\end{Theorem} 

 We now discuss the problem of finding a criterion for the $A_1(\cD)$-convergence of WGA($\tau$). Clearly, sufficient conditions from Theorems \ref{DT1} and \ref{DT2} from above are also sufficient conditions for the $A_1(\cD)$-convergence of WGA($\tau$) for any $\cD$. It turns out that the counterexamples, which were constructed in Theorems \ref{DT1} and \ref{DT2} in the proofs of necessity, have bounded $A_1(\cD)$ norms. This implies the following results for the $A_1(\cD)$-convergence.
 
 \begin{Theorem}\label{DT3} The following condition on the weakness sequences $\tau$ 
\be\label{D3}
\forall \{a_j\} \in \ell_2,\quad a_j\ge 0, j=1,2,\dots,\quad\liminf_{n\to \infty}\frac{a_{n}}{t_{n}}\sum_{j=1}^{n}a_j =0
\ee
is the necessary and sufficient condition for the $A_1(\cD)$-convergence of the WGA($\tau$) for all dictionaries $\cD$.
\end{Theorem} 
 
\begin{Theorem}\label{DT4} Assume that the weakness sequences $\tau =\{t_k\}_{k=1}^\infty$ are monotone $t_1\ge t_2 \ge \dots$. Then, the following condition on the weakness sequences~$\tau$ 
\be\label{D4}
\sum_{k=1}^\infty \frac{t_k}{k} =\infty
\ee
is the necessary and sufficient condition for the $A_1(\cD)$-convergence of the WGA($\tau$) for all dictionaries $\cD$.
\end{Theorem} 

We stress that in Theorems \ref{DT3} and \ref{DT4}  we talk about conditions, which guarantee the $A_1(\cD)$-convergence of the WGA($\tau$) for all dictionaries $\cD$. Theorems \ref{DT1} -- \ref{DT4} show (surprisingly) that in the case of WGA($\tau$) the criteria for convergence and for the $A_1(\cD)$-convergence for all dictionaries $\cD$ are the same.

We now make a comment on Theorem \ref{TT1} (also see Remark \ref{TR1a}). This theorem claims that in the case of 
$\cD$ being an orthonormal basis the condition
\be\label{D5}
\sum_{k=1}^\infty t_k =\infty
\ee
gives the criterion for the $A_1(\cD)$-convergence. It is well known (see, for instance, \cite{VTbook}, p.84) and easy to see 
that the condition
\be\label{D6}
\sum_{k=1}^\infty t_k ^2=\infty
\ee
gives the criterion for convergence of the WGA($\tau$) and two other algorithms mentioned in Remark \ref{TR1a} in the case 
of $\cD$ being an orthonormal basis. Clearly, (\ref{D5}) and (\ref{D6}) are very different conditions. 

The most important and technically involved results of this paper are in Sections \ref{tb} and \ref{BS}. On one hand results in those sections are obtained under an extra assumption that the weakness sequence is monotone decreasing. It would be interesting to obtain results on the rate of convergence without the monotonicity assumption. For instance, in the case
$\tau = \tau(\cN,t)$. On the other hand those results cannot be substantially improved. We demonstrate it on Theorem \ref{RCT}.   Corollary \ref{PC1} implies that under the monotonicity assumption we have for the WGA($\tau$):
For $f \in A_1(\cD)$  
\be\label{D7}
\|f_m\| \le \left(1+\sum_{k=1}^m t^2_k\right)^{-\frac{t_m}{2(2+t_m)}}.  
\ee
In a particular case $\tau =\{t\}$ (\ref{D7}) gives
$$
\|f_m\| \le (1+mt^2)^{-\frac{t}{4+2t}}, \quad 0<t\le 1. 
$$
This estimate implies the  inequality 
$$
\|f_m\| \le C_1(t)m^{-at},   
$$
with the exponent $at$ approaching $0$ linearly in $t$. It was proved in \cite{LTe2}   that this exponent cannot decrease to $0$ at a slower rate than linear.

\begin{Theorem}[{\cite{LTe2}}]\label{DT5}   There exists an absolute constant $b>0$ such that, for any $t>0$, we can find a dictionary $\cD_t$ and a function $f_t \in A_1(\cD_t)$ such that, for some realization $G_m^t(f_t,\cD_t)$ of the Weak Greedy Algorithm with weakness parameter $t$, we have
$$
\liminf_{m\to \infty} \|f_t -G_m^t(f_t,\cD_t)\|m^{bt}  >0. 
$$
\end{Theorem}

 \Addresses
 
\end{document}